\documentclass[12pt]{amsart}
\usepackage{amsfonts}
\usepackage{amsfonts,latexsym,rawfonts,amsmath,amssymb,amsthm}
\usepackage[plainpages=false]{hyperref}
\usepackage{graphicx}

\numberwithin{equation}{section}

\RequirePackage{color}
\theoremstyle{plain}

\newtheorem{theorem}{Theorem}[section]

\newtheorem{corollary}{Corollary}[section]
\newtheorem{definition}{Definition}[section]

\newtheorem{lemma}[theorem]{Lemma}
\newtheorem{proposition}[theorem]{Proposition}
\newtheorem{remark}[theorem]{Remark}
\newtheorem{problem}[theorem]{Problem}

\newcommand{\beq}{\begin{equation}}
\newcommand{\eeq}{\end{equation}}
\newcommand{\beqs}{\begin{eqnarray*}}
\newcommand{\eeqs}{\end{eqnarray*}}
\newcommand{\beqn}{\begin{eqnarray}}
\newcommand{\eeqn}{\end{eqnarray}}
\newcommand{\beqa}{\begin{array}}
\newcommand{\eeqa}{\end{array}}

\def\phi{\varphi}

\begin{document}
\title[The Christoffel problem]{Smooth solutions to the Christoffel problem in $\mathbb{H}^{n+1}$}

\author{Li Chen}
\address{Faculty of Mathematics and Statistics, Hubei Key Laboratory of Applied Mathematics, Hubei University,  Wuhan 430062, P.R. China}
\email{chernli@163.com}




\begin{abstract}
The famous Christoffel problem is possibly the oldest problem of prescribed curvatures
for convex hypersurfaces in Euclidean space. Recently, this problem has been
naturally formulated in the context of uniformly $h$-convex hypersurfaces in hyperbolic space
by Espinar-G\'alvez-Mira.

Surprisingly, Espinar-G\'alvez-Mira find that the Christoffel problem in hyperbolic space
is essentially equivalent to the Nirenberg-Kazdan-Warner problem on prescribing
scalar curvature on $\mathbb{S}^n$. This equivalence opens a new door to
study the Nirenberg-Kazdan-Warner problem.

In this paper, we establish a existence of solutions to
the Christoffel problem in hyperbolic space by
proving a full rank theorem.
As a corollary, a existence of solutions to the Nirenberg-Kazdan-Warner problem follows.
\end{abstract}

\maketitle

\baselineskip18pt

\parskip3pt

\section{Introduction}

The famous Christoffel problem is possibly the oldest problem of prescribed curvatures
for convex hypersurfaces in Euclidean space \cite{Chr65}. Recently, this problem has been
naturally formulated in the context of uniformly $h$-convex hypersurfaces in hyperbolic space
by Espinar-G\'alvez-Mira \cite{Esp09}. See Section 2 for a comprehensive introduction.
From PDEs view, the Christoffel problem concerns that  given a positive function $f$
defined on $\mathbb{S}^n$, can one find a uniformly $h$-convex solution to the following equation
\begin{eqnarray}\label{MA}
\Delta\varphi-\frac{n}{2}\frac{|D\varphi|^2}{\varphi}+\frac{n}{2}\Big(\varphi-\frac{1}{\varphi}\Big)=\varphi^{-1}f(x),
\end{eqnarray}
where $D$ and $\Delta$ are the covariant derivative and Laplace operator
on $\mathbb{S}^n$ respectively, and a uniformly $h$-convex solution refers to a solution $\varphi$ with the matrix
\begin{eqnarray*}
U[\varphi]=D^2\varphi-\frac{1}{2}\frac{|D\varphi|^2}{\varphi}I+\frac{1}{2}\Big(\varphi-\frac{1}{\varphi}\Big)I
\end{eqnarray*}
is positive definite everywhere on $\mathbb{S}^n$, where $I$ is the identity matrix.

However, Espinar-G\'alvez-Mira do not provide a solution to the Christoffel problem \eqref{MA}.
We prove a existence of solutions to \eqref{MA}.

\begin{theorem}\label{Main}
For a given smooth and positive function $f$ defined on $\mathbb{S}^n$, if $f$ satisfies

(1) $f$ is an even function, i.e.
$f(x)=f(-x)$ for all $x \in \mathbb{S}^n$,

(2) the matrix
\begin{eqnarray*}
D^2(f^{-1})-|D f^{-1}|I
+\frac{f^{-1}}{\frac{8}{n}\max_{\mathbb{S}^n}f+2}I
\end{eqnarray*}
is positive semi-definite on $\mathbb{S}^n$,

then, there exists a smooth, even and uniformly $h$-convex solution to Problem \eqref{MA}.
\end{theorem}

\begin{remark}
For a given smooth and positive function $h(x)$ defined on $\mathbb{S}^n$, it is easy to see that
$f=(h^{-1}+C)^{-1}$ satisfies Condition (2) in Theorem \ref{Main} for $C$ large enough.
Thus, the class of functions which satisfy Condition (2) is quite large.
\end{remark}

Since the uniformly $h$-convex solution of \eqref{MA} is not its admissible solution, the proof of
Theorem \ref{Main} heavily relies on a full rank theorem (Theorem \ref{FRT}) which is proved
following Guan-Ma's work on the Christoffel-Minkowski problem (see Theorem 1.2 in \cite{GM03}).

Surprisingly, Espinar-G\'alvez-Mira showed that the Christoffel problem in $\mathbb{H}^{n+1}$
is essentially equivalent to the Nirenberg-Kazdan-Warner problem on prescribing
scalar curvature on $\mathbb{S}^n$ (see Theorem 15 in \cite{Esp09} or Section 5 in \cite{And20}).

\begin{problem}[Prescribing Scalar Curvature on $\mathbb{S}^n$]
Let $(\mathbb{S}^n, g_0)$ be the standard $n$-sphere,
one may ask which function $R(x)$ on $\mathbb{S}^n$ is the scalar curvature of a metric $g$ on
$\mathbb{S}^n$ conformally equivalent to $g_0$?
For $n\geq 3$, we write $g=u^{\frac{4}{n-2}}g_0$, the problem is equivalent to finding a solution of
\begin{eqnarray}\label{NI}
-\Delta v+\frac{n(n-2)}{4}v=\frac{n-2}{4(n-1)}R(x)v^{\frac{n+2}{n-2}},
\quad v>0, \quad \mbox{on} \quad \mathbb{S}^n.
\end{eqnarray}
\end{problem}

It is easy to check that if we set $\varphi=v^{-\frac{2}{n-2}}$, then
the Christoffel problem \eqref{MA} becomes the Nirenberg-Kazdan-Warner problem
\eqref{NI} with $R=n(n-1)(f+\frac{1}{2})$.
Thus, as a corollary of Theorem \ref{Main}, we get a existence of solutions to Problem \eqref{NI}.

\begin{corollary}
For a given smooth and positive function $f$ defined on $\mathbb{S}^n$ with $n\geq 3$, if $f$ satisfies
Conditions (1) and (2) in Theorem \ref{Main},
then, there exists a smooth solution to Problem \eqref{NI} with $R=n(n-1)(f+\frac{1}{2})$.
\end{corollary}

For a comprehensive survey on the Nirenberg-Kazdan-Warner problem, see \cite{Li02}.

In Sect.2, we will give an introduction to the Christoffel problem in $\mathbb{H}^{n+1}$.
After establishing the a priori estimates for solutions to
the Christoffel problem in $\mathbb{H}^{n+1}$ in Sect.3 and
the full rank theorem in Sect.4, we will use the degree theory to prove Theorem \ref{Main} in Sect.5.

\section{Preliminaries}

\subsection{The Christoffel problem in $\mathbb{H}^{n+1}$}

For a smooth, closed and uniformly convex hypersurface $\Sigma$ in $\mathbb{R}^{n+1}$, the Gauss
map $n: \Sigma\rightarrow \mathbb{S}^n$ is a diffeomorphism.
The famous Christoffel problem concerns that  given a positive function $f$
defined on $\mathbb{S}^n$, can one find
a closed and uniformly convex hypersurface $\Sigma$ in $\mathbb{R}^{n+1}$ with $f$ as the mean of its curvature radii \cite{Chr65}:
\begin{eqnarray*}
\sum_{i=1}^{n}R_i(n^{-1}(x))=f(x), \quad \forall x \in \mathbb{S}^n,
\end{eqnarray*}
where $R_i:=\kappa_{i}^{-1}$
are the curvature radii of the hypersurface and $\kappa_1, ..., \kappa_n$ are the principal curvatures of the hypersurface.
This problem was classically solved in \cite{Fi1, Fi2}.

It is very natural to ask for similar problems above in space forms.
Espinar-G\'alvez-Mira \cite{Esp09} first show that the Christoffel problem can be
naturally formulated in the context of uniformly $h$-convex hypersurfaces in the hyperbolic space.
In order to describe how they formulate this problem, at first, we will briefly describe $h$-convex geometry in hyperbolic space
followed by Sections 2 and 3 in \cite{Esp09} (see also Section 5 in \cite{And20} and Section 2 in \cite{Li-Xu}). Their works deeply declared
interesting formal similarities between the geometry of
$h$-convex domains in hyperbolic space and that of
convex Euclidean bodies.

We shall work
in the hyperboloid model of $\mathbb{H}^{n+1}$. For that, consider the
Minkowski space $\mathbb{R}^{n+1,1}$ with canonical coordinates
$X=(X^1, . . . , X^{n+1}, X^0)$ and the Lorentzian metric
\begin{eqnarray*}
\langle X, X\rangle=\sum_{i=1}^{n+1}(X^i)^2-(X^0)^2.
\end{eqnarray*}
$\mathbb{H}^{n+1}$ is the future time-like hyperboloid in Minkowski
space $\mathbb{R}^{n+1,1}$, i.e.
\begin{eqnarray*}
\mathbb{H}^{n+1}=\Big\{X=(X^1, \cdot\cdot\cdot, X^{n+1}, X^0) \in
\mathbb{R}^{n+1,1}: \langle X, X\rangle=-1, X^0>0\Big\}.
\end{eqnarray*}
Given two points $X$ and $Y$ on $\mathbb{H}^{n+1}$, the geodesic distance between $X$ and $Y$
is denoted by $d_{\mathbb{H}^{n+1}}(X, Y)$. For any $X \in \mathbb{H}^{n+1}$ with
$r=\mathrm{dist}((0, 1), X)$, there exists $\theta \in \mathbb{S}^n$ such that
\begin{eqnarray}\label{X-1}
X=(\sinh r\theta, \cosh r).
\end{eqnarray}

The horospheres are hypersurfaces in $\mathbb{H}^{n+1}$ whose
principal curvatures equal to $1$ everywhere. In the hyperboloid model of $\mathbb{H}^{n+1}$,
they can be parameterized by $\mathbb{S}^n \times \mathbb{R}$
\begin{eqnarray*}
H_{x}(r)=\{X \in \mathbb{H}^{n+1}: \langle X, (x, 1)\rangle=-e^r\},
\end{eqnarray*}
where $x \in \mathbb{S}^n$ and $r \in \mathbb{R}$ represents the signed geodesic distance from the ``north pole" $N =(0, 1)\in
\mathbb{H}^{n+1}$.
The interior of the horosphere is called the horo-ball and we denote by
\begin{eqnarray*}
B_{x}(r)=\{X \in \mathbb{H}^{n+1}: 0>\langle X, (x,
1)\rangle>-e^r\}.
\end{eqnarray*}
If we use the Poincar\'e ball model $\mathbb{B}^{n+1}$ of $\mathbb{H}^{n+1}$, then $B_x(r)$ corresponds to
an $(n+1)$-dimensional ball which tangents to $\partial \mathbb{B}^{n+1}$ at $x$. Furthermore, $B_x(r)$ contains the origin for $r>0$.
\begin{definition}
A compact domain $\Omega\subset \mathbb{H}^{n+1}$ (or its boundary $\partial \Omega$) is horospherically convex
(or $h$-convex for short) if every boundary point $p$ of
$\partial \Omega$ has a supporting horo-ball, i.e. a horo-ball $B$
such that $\Omega\subset \overline{B}$ and $p \in \partial B$. When $\Omega$ is smooth, it is $h$-convex if
and only if the principal curvatures of $\partial \Omega$ are greater
than or equal to $1$.

For a smooth compact domain $\Omega$, we say
$\Omega$ (or $\partial \Omega$) is uniformly $h$-convex if its principal curvatures are greater than $1$.
\end{definition}

\begin{definition}\label{HG}
Let $\Omega\subset \mathbb{H}^{n+1}$ be a $h$-convex compact domain. For each $X \in \partial \Omega$,
$\partial \Omega$ has a supporting horo-ball $B_{x}(r)$
for some $r \in \mathbb{R}$ and $x \in \mathbb{S}^n$. Then the horospherical
Gauss map $G: \partial \Omega\rightarrow S_{\infty}=(\mathbb{S}^n, g_{\infty})$ of $\Omega$ (or $\partial \Omega$)
is defined by
\begin{eqnarray*}
G(X)=x, \quad g_{\infty}(X)=e^{2r}\sigma,
\end{eqnarray*}
where $\sigma$ is the canonical $\mathbb{S}^n$ metric.
\end{definition}

Note that the canonical $\mathbb{S}^n$ metric $\sigma$ is used in order to measure geometric quantities associated to the
Euclidean Gauss map of a hypersurface in $\mathbb{R}^{n+1}$. However, Espinar-G\'alvez-Mira
explained in detail why we use the horospherical metric $g_{\infty}$ on $\mathbb{S}^n$
for measuring geometrical quantities with respect to the horospherical Gauss map (see section 2 in \cite{Esp09}).
Let $M$ be a $h$-convex hypersurface in $\mathbb{H}^{n+1}$.
For each $X \in M$, $M$  has a supporting horo-ball $B_{x}(r)$.
Then, we have (see (5.3) in \cite{And20} or (2.2) in \cite{Li-Xu})
\begin{eqnarray}\label{X-v}
X-\nu=e^{-r}\langle x, 1\rangle,
\end{eqnarray}
where $\nu$ is the unit outward vector of $M$. Differentiating the equation \eqref{X-v} gives
\begin{eqnarray}\label{dG}
\langle d G, d G\rangle_{g_{\infty}}=(\mathcal{W}-I)^2,
\end{eqnarray}
where $\mathcal{W}$ is Weingarten matrix of $M$ and $I$ is the identity matrix. For a smooth and uniformly $h$-convex hypersurface $M\subset \mathbb{H}^{n+1}$, the matrix $\mathcal{W}-I$ is positive definite. Thus, $G$ is a diffeomorphism from $M$ to $\mathbb{S}^n$.

The relation \eqref{dG} motivated Espinar-G\'alvez-Mira to define by the hyperbolic curvature radii (see Definition 8 in \cite{Esp09})
\begin{eqnarray}\label{HCR}
\mathcal{R}_i(p):=\frac{|e_i(p)|}{|d_{e_i}G(p)|}=\frac{1}{\kappa_i(p)-1}, \quad p \in M,
\end{eqnarray}
where $\{e_1(p), ... , e_n(p)\}$ is an orthonormal basis of principal directions of $T_p M$. Thus, Espinar-G\'alvez-Mira \cite{Esp09}
proposed the following Christoffel problem in $\mathbb{H}^{n+1}$.

\begin{problem}[The Christoffel problem]\label{C}
Let $f: \mathbb{S}^n\rightarrow \mathbb{R}$. Find if
there is a smooth, closed and uniformly $h$-convex hypersurface such that for every $x \in \mathbb{S}^n$
\begin{eqnarray*}
\sum_{i=1}^{n}\mathcal{R}_i(G^{-1}(x))=f(x),
\end{eqnarray*}
where $G:\mathbb{S}^n\rightarrow \mathbb{H}^{n+1}$ is the horospherical Gauss map of the uniformly $h$-convex hypersurface
and $\mathcal{R}_1, \mathcal{R}_2, ..., \mathcal{R}_n$ are the hyperbolic curvature radii \eqref{HCR}.
\end{problem}

Using the the horospherical support function of a
$h$-convex hypersurface in $\mathbb{H}^{n+1}$ which will be defined below, we can reduce Problem \ref{C} to solve a Laplace
equation on $\mathbb{S}^n$.
Let $\Omega$ be a $h$-convex compact domain in $\mathbb{H}^{n+1}$. Then
for each $x\in \mathbb{S}^n$ we define the horospherical support function
of $\Omega$ (or $\partial \Omega$) in direction $x$ by
\begin{eqnarray*}
u(x):=\inf\{s \in \mathbb{R}: \Omega\subset \overline{B}_{x}(s)\}.
\end{eqnarray*}
We also have the alternative characterisation
\begin{eqnarray}\label{SD}
u(x)=\sup\{\log(-\langle X, (x, 1)\rangle): X \in \Omega\}.
\end{eqnarray}
The support function completely determines a
$h$-convex compact domain $\Omega$, as an intersection of horo-balls:
\begin{eqnarray*}
\Omega=\bigcap_{x \in \mathbb{S}^n}\overline{B}_{x}(u(x)).
\end{eqnarray*}

Since $G$ is a diffeomorphism from $\partial \Omega$ to $\mathbb{S}^n$ for a compact uniformly $h$-convex domain $\Omega$, then
$\overline{X}=G^{-1}$ is a smooth embedding from $\mathbb{S}^n$ to $\partial \Omega$ and
$\overline{X}$ can be written in terms of the support function $u$, as follows:
\begin{eqnarray}\label{X}
\overline{X}(x)=\frac{1}{2}\varphi(-x, 1)+\frac{1}{2}
\Big(\frac{|D \varphi|^2}{\varphi}+\frac{1}{\varphi}\Big)(x, 1)-(D\varphi, 0),
\end{eqnarray}
where $\varphi=e^u$ and
$D$ is the Levi-Civita connection of the standard metric $\sigma$ of $\mathbb{S}^{n}$. Then,
we have by \eqref{X-v}
\begin{eqnarray}\label{X-v-1}
X-\nu=\frac{1}{\varphi}\langle x, 1\rangle.
\end{eqnarray}
and after choosing normal coordinates around $x$ on $\mathbb{S}^{n}$, we
express the Weingarten matrix in the horospherical support function
(see (1.16) in \cite{And20}, Lemma 2.2 in \cite{Li-Xu})
\begin{eqnarray}\label{w-i}
W-I=\Big(\varphi U[\varphi]\Big)^{-1},
\end{eqnarray}
where
\begin{eqnarray*}
U[\varphi]=D^2\varphi-\frac{1}{2}\frac{|D\varphi|^2}{\varphi}I+\frac{1}{2}\Big(\varphi-\frac{1}{\varphi}\Big)I.
\end{eqnarray*}
Clearly, $\Omega \subset \mathbb{H}^{n+1}$ is  uniformly $h$-convex if and only if
the matrix $U[\varphi]$ is positive definite. Thus, the hyperbolic curvature radii $\mathcal{R}_i$ are
the eigenvalues of the matrix $\varphi U[\varphi]$. So, The Christoffel problem \ref{C} is equivalent to finding a uniformly $h$-convex solution to the equation \eqref{MA}:
\begin{eqnarray*}
\Delta\varphi-\frac{n}{2}\frac{|D\varphi|^2}{\varphi}+\frac{n}{2}\Big(\varphi-\frac{1}{\varphi}\Big)=\varphi^{-1}f(x),
\end{eqnarray*}
where a uniformly $h$-convex solution refers to a solution $\varphi$ with the matrix $U[\varphi]$
is positive definite everywhere on $\mathbb{S}^n$.

\subsection{The elementary symmetric functions}

We will give the definition of the elementary symmetric functions and review their basic properties which could be found in
\cite{L96}.

\begin{definition}
For any $k=1,2,\cdots,n$, we set
\begin{eqnarray}\label{sigma}
\sigma_{k}(\lambda)=\sum\limits_{1\le i_1<i_2<\cdots < i_k\le n}
\lambda_{{i}_{1}}\lambda_{{i}_{2}}\cdots\lambda_{{i}_{k}},
\end{eqnarray}
for any $\lambda=(\lambda_{1},\cdots,\lambda_{n})\in\mathbb{R}^{n}$ and set $\sigma_0(\lambda)=1$.
Let $\lambda_1(A)$, ..., $\lambda_n(A)$ be the eigenvalues of the $n \times n$ symmetric matrix $A$ and
denote by $\lambda(A)=(\lambda_1(A), ..., \lambda_n(A)$. We define by $\sigma_{k}(A)=\sigma_{k}(\lambda(A))$.
\end{definition}

We recall that the Garding's cone is defined as
$$\Gamma_{k}=\{\lambda\in \mathbb{R}^{n}:\sigma_{i}(\lambda)>0,\forall 1\le i\le k\}.$$

\begin{proposition}\label{AA}
Let $A$ be a $n \times n$ symmetric matrix. Then the following relations hold.
\begin{eqnarray*}
\sigma_k(A)=\frac{1}{k!}\sum_{\substack{i_1, ..., i_k=1\\ j_1, ..., j_k=1}}^{n}
\delta(i_1, ..., i_k; j_1, ..., j_k)A_{i_1 j_1}\cdot \cdot \cdot A_{i_k j_k},
\end{eqnarray*}
\begin{eqnarray*}
\sigma_{k}^{\alpha\beta}(A)&:=&\frac{\partial\sigma_{k}}{\partial A_{\alpha\beta}}(A)
\\&=&
\frac{1}{(k-1)!}\sum_{\substack{i_1, ..., i_k=1\\ j_1, ..., j_k=1}}^{n}
\delta(\alpha, i_1, ..., i_{k-1}; \beta, j_1, ..., j_{k-1})A_{i_1 j_1}\cdot \cdot \cdot A_{i_{k-1} j_{k-1}},
\end{eqnarray*}
\begin{eqnarray*}
\sigma_{k}^{\alpha\beta, \mu\nu}(A)&:=&\frac{\partial^2\sigma_{k}}{\partial A_{\alpha\beta}\partial A_{\mu\nu}}(A)
\\&=&
\frac{1}{(k-2)!}\sum_{\substack{i_1, ..., i_k=1\\ j_1, ..., j_k=1}}^{n}
\delta(\alpha, \mu, i_1, ..., i_{k-2}; \beta, \nu, j_1, ..., j_{k-2})A_{i_1 j_1}\cdot \cdot \cdot A_{i_{k-2} j_{k-2}},
\end{eqnarray*}
where the Kronecker symbol $\delta(I; J)$ for indices $I=(i_1, ..., i_m)$ and $J=(j_1, ..., j_m)$ is defined as
\begin{equation*}
\delta(I; J)=
  \begin{cases}
    \displaystyle 1, & \text{if I is an even permutation of J};\\[2.5ex]
    \displaystyle -1, & \text{if I is an odd permutation of J};\\[2.5ex]
    \displaystyle 0, & \text{otherwise}.
  \end{cases}
\end{equation*}
\end{proposition}

For $h$-convex hypersurfaces in $\mathbb{H}^{n+1}$, Andrews-Chen-Wei \cite{And20} introduce
the shifted Weingarten matrix of hypersurfaces
$\widetilde{\mathcal{W}}:=\mathcal{W}-I$ based on
the relation \eqref{dG}.
Let $\tilde{\kappa}=(\tilde{\kappa}_1, . . . , \tilde{\kappa}_n)$ are
eigenvalues of the shifted Weingarten matrix $\widetilde{\mathcal{W}}$, they define the $k$-th shifted mean curvature
\begin{eqnarray*}
\widetilde{H}_k:=\frac{1}{C_{n}^{k}}\sigma_k(\tilde{\kappa}).
\end{eqnarray*}
The geometry and analysis on shifted curvatures of hypersurfaces in hyperbolic space
have been widely studied in \cite{And20, Hu20, Wang20, Hu23} recently.

We recall Lemma 2.6 in \cite{Hu20} which gives the relation of shifted mean curvatures.
\begin{lemma}
Let $M$ be a smooth closed hypersurface in $H^{n+1}$. Then
\begin{eqnarray}\label{QQ}
\int_{M}(\cosh r-\langle V, \nu\rangle) \widetilde{H}_k d\mu=\int_{M}\langle V, \nu\rangle \widetilde{H}_{k+1} d\mu,
\end{eqnarray}
where $r$ is the radial function of $M$, $V=\sinh r\partial_r$ be the conformal Killing vector field
in $\mathbb{H}^{n+1}$, $\nu$ is the unit outward normal vector of $M$ and $d\mu$ is the area element of $M$.
\end{lemma}

For uniformly $h$-convex hypersurfaces in $\mathbb{H}^{n+1}$, we can express the equality
\eqref{QQ} by the horospherical support function of hypersurfaces.

\begin{lemma}
Let $M$ be a smooth, closed and uniformly $h$-convex hypersurface in $\mathbb{H}^{n+1}$
and $u=\log \varphi$ be its horospherical support function. Then
\begin{eqnarray}\label{DE}
&&\int_{\mathbb{S}^n}\varphi^{k-n}\sigma_{k+1}(U[\varphi])d\sigma_{\mathbb{S}^n}
\nonumber\\&=&\frac{C_{n}^{k}}{C_{n}^{k+1}}\int_{\mathbb{S}^n}\Big[\frac{|D \varphi|^2}{2\varphi}+\frac{1}{2}(\varphi-\varphi^{-1})\Big]
\varphi^{k-n}\sigma_{k}(U[\varphi])d\sigma_{\mathbb{S}^n},
\end{eqnarray}
where $d\sigma_{\mathbb{S}^n}$ is the area element of $\mathbb{S}^n$.
\end{lemma}

\begin{proof}
Comparing \eqref{X-1} with \eqref{X}, we find
\begin{eqnarray}\label{SS-1}
\cosh r=\frac{|D \varphi|^2}{2\varphi}+\frac{1}{2}(\varphi+\varphi^{-1})
\end{eqnarray}
and
\begin{eqnarray}\label{SS-2}
\sinh r\langle \theta, x\rangle=\frac{|D \varphi|^2}{2\varphi}-\frac{1}{2}(\varphi-\varphi^{-1}).
\end{eqnarray}
Note that $\partial_r=\partial_r X=(\cosh r\theta, \sinh r)$ and
$\langle\partial_r, X\rangle=0$. Using \eqref{X-v-1}, we have
\begin{eqnarray*}
\langle X, \nu\rangle=-\sinh r\langle \partial_r, X-\nu\rangle
=\varphi^{-1}(\sinh^2 r-\cosh r \sinh r\langle \theta, x\rangle).
\end{eqnarray*}
Plugging \eqref{SS-1} and \eqref{SS-2} into the above equation, it yields
\begin{eqnarray}\label{SS-3}
\langle V, \nu\rangle=\frac{|D \varphi|^2}{2\varphi}+\frac{1}{2}(\varphi-\varphi^{-1}).
\end{eqnarray}

Combining \eqref{dG} and \eqref{w-i}, we have
\begin{eqnarray*}
d\mu=\det U[\varphi]d\sigma_{\mathbb{S}^n}.
\end{eqnarray*}
Thus,
\begin{eqnarray}\label{SS-4}
\widetilde{H}_k d\mu=\frac{1}{C_{n}^{k}}\varphi^{-k}\sigma_{n-k}(U[\varphi])d\sigma_{\mathbb{S}^n}.
\end{eqnarray}
Inserting \eqref{SS-1}, \eqref{SS-3} and \eqref{SS-4} into \eqref{QQ}, then \eqref{DE} follows.
\end{proof}

\section{The a priori estimates}

For convenience, in the following of this paper, we always assume that $f$ is
a smooth positive, even function on $\mathbb{S}^n$. Let $M$ be the smooth, closed,
origin-symmetric and uniformly $h$-convex hypersurface in
$\mathbb{H}^{n+1}$ with the horospherical support function $u=\log \varphi$. Assume
$\varphi$ is a smooth solution
to the equation \eqref{MA}. Clearly,
$\varphi$ is a smooth even and uniformly $h$-convex solution
to the equation \eqref{MA} and $\varphi(x)> 1$ for $x \in \mathbb{S}^n$.

The following easy and important equality is key for
the $C^0$ estimate.
\begin{lemma}
We have
\begin{eqnarray}\label{c0-12}
\frac{1}{2}\Big(\max_{\mathbb{S}^n}\varphi+\frac{1}{\max_{\mathbb{S}^n}\varphi}\Big)\leq \min_{\mathbb{S}^n}\varphi.
\end{eqnarray}
\end{lemma}

\begin{proof}
The inequality can be found in the proof of Lemma 7.2 in \cite{Li-Xu}. For completeness, we give a proof here.
Assume that $\varphi(x_1)=\max_{\mathbb{S}^n}\varphi$ and denote $\overline{X}(x_1)=G^{-1}(x_1)$ as before. Then, we have for any $x \in \mathbb{S}^n$ by
the definition of the horospherical support function \eqref{SD}
\begin{eqnarray*}
-\langle \overline{X}(x_1), (x, 1)\rangle\leq \varphi(x), \quad \forall x \in \mathbb{S}^n.
\end{eqnarray*}
Substituting the expression \eqref{X} for $\overline{X}$ into the above equality yields
\begin{eqnarray}\label{D0-3}
\frac{1}{2}\varphi(x_1)(1+\langle x_1, x\rangle)+\frac{1}{2}\frac{1}{\varphi(x_1)}(1-\langle x_1, x\rangle)\leq \varphi(x),
\end{eqnarray}
where we used the fact $D\varphi(x_1)=0$. Note that $\varphi(x_1)\geq 1$, we find from \eqref{D0-3}
\begin{eqnarray}\label{D0-4}
\frac{1}{2}\Big(\varphi(x_1)+\frac{1}{\varphi(x_1)}\Big)\leq \varphi(x) \quad \mbox{for} \quad \langle x, x_1\rangle\geq 0.
\end{eqnarray}
Since $\varphi$ is even, we can assume
that the minimum point $x_0$ of $\varphi(x)$ satisfies $\langle x_0, x_1\rangle\geq 0$. Thus,
the equality \eqref{c0-12} follows that from \eqref{D0-4}.
\end{proof}

Now, we use the maximum principle to get the $C^0$-estimate.
\begin{lemma}\label{C-C0}
We have
\begin{eqnarray}\label{C0}
1<C\leq\min_{\mathbb{S}^n}\varphi\leq \max_{\mathbb{S}^n}\varphi\leq
\Big(\frac{2}{n}\max_{\mathbb{S}^n}f+1\Big)^{\frac{1}{2}}+\Big(\frac{2}{n}\max_{\mathbb{S}^n}f\Big)^{\frac{1}{2}},
\end{eqnarray}
where $C$ is a positive constant depending on, $n$ and $\min_{\mathbb{S}^n}f$.
\end{lemma}

\begin{proof}
Applying the maximum principle, we have from the equation \eqref{MA}
\begin{eqnarray}\label{c0-11}
\min_{\mathbb{S}^n}\varphi\leq \Big(\frac{2}{n}\max_{\mathbb{S}^n}f+1\Big)^{\frac{1}{2}}, \quad \mbox{and} \quad
\max_{\mathbb{S}^n}\varphi\geq \Big(\frac{2}{n}\min_{\mathbb{S}^n}f+1\Big)^{\frac{1}{2}}>1.
\end{eqnarray}
Combining \eqref{c0-11} and \eqref{c0-12}, we find
\begin{eqnarray*}
1<C\leq\min_{\mathbb{S}^n}\varphi\leq \max_{\mathbb{S}^n}\varphi\leq \Big(\frac{2}{n}\max_{\mathbb{S}^n}f+1\Big)^{\frac{1}{2}}+\Big(\frac{2}{n}\max_{\mathbb{S}^n}f\Big)^{\frac{1}{2}}.
\end{eqnarray*}
So, we complete the proof.
\end{proof}

To obtain the the gradient estimate, we recall Lemmas 7.3 in \cite{Li-Xu}.

\begin{lemma}
For a smooth, origin symmetric and uniformly $h$-convex hypersurface $M$ in $\mathbb{H}^{n+1}$, we have
\begin{eqnarray}\label{r-u}
|D \varphi|<\varphi \quad \mbox{on} \quad \mathbb{S}^n.
\end{eqnarray}
\end{lemma}

The gradient estimate follows from \eqref{r-u}.

\begin{lemma}\label{C-C1}
We have
\begin{eqnarray}\label{C1}
|D\varphi(x)|\leq C, \quad \forall \ x \in \mathbb{S}^n,
\end{eqnarray}
where $C$ is a positive constant depending only on the constant in Lemma \ref{C-C0}.
\end{lemma}

Higher estimates follows from Schauder estimates and the positivity of the matrix $U[\varphi]$.

\begin{lemma}
We have
\begin{eqnarray}\label{C2+}
|\varphi|_{C^{4,\alpha}(\mathbb{S}^n)}\leq C,
\end{eqnarray}
where $C$ is a positive constant depending
only on the constant in Lemmas \ref{C-C0} and \ref{C-C1}.
\end{lemma}

\begin{proof}
Let $\lambda_1(U), ..., \lambda_n(U)$ are eigenvalues of the matrix $U$.
On the one hand, we have that for $1\leq i\leq n$ by the positivity of the matrix $U$
and $C^0$ estimate \eqref{C0}
\begin{eqnarray}\label{o2}
\lambda_i(U)\leq \mathrm{tr} U=\Delta \varphi-\frac{n}{2\varphi}|D\varphi|^2
+\frac{n}{2}(\varphi-\frac{1}{\varphi})=\varphi^{-1}f\leq C.
\end{eqnarray}
On the other hand, we have that for $1\leq i\leq n$ by the positivity of the matrix $U$, \eqref{o2}
and $C^0$ estimate \eqref{C0}
\begin{eqnarray*}
\lambda_i(U)=\mathrm{tr} U-\sum_{j=1, j\neq i}^{n}\lambda_j(U)\geq \varphi^{-1}f-(n-1)C\geq -C.
\end{eqnarray*}
Thus,
\begin{eqnarray*}
|\varphi|_{C^{2}(\mathbb{S}^n)}\leq C.
\end{eqnarray*}
Therefore, the a priori estimate \eqref{C2+} follows from Schauder estimates \cite{GT}.
\end{proof}

\section{Full rank theorem}

In this section, we prove the following full rank theorem.

\begin{theorem}\label{FRT}
Let $f$ be a smooth and positive function on $\mathbb{S}^n$ satisfying Conditions (1) and (2)
in Theorem \ref{Main}.
If $\varphi$ is an even and $h$-convex solution to the equation \eqref{MA} satisfying \eqref{DE}, then it is an even and
uniformly $h$-convex solution to the equation \eqref{MA}, where
a $h$-convex solution $\varphi$ is a solution satisfying that $U[\varphi]$ is
positive semi-definite on $\mathbb{S}^n$.
\end{theorem}

Recall
\begin{eqnarray*}
U[\varphi]=D^2\varphi-\frac{1}{2}\frac{|D\varphi|^2}{\varphi}I+\frac{1}{2}(\varphi-\frac{1}{\varphi})I.
\end{eqnarray*}

\begin{lemma}
For each fixed $x \in \mathbb{S}^n$, we choose a local orthonormal frame so that $U$ is diagonal at $x$,
then we have at $x$
\begin{eqnarray}\label{1ex-0}
U_{ij\alpha}=U_{ji \alpha}, \quad i \neq j,
\end{eqnarray}

\begin{eqnarray}\label{1ex}
U_{\alpha\alpha i}=U_{i \alpha \alpha}-\frac{\varphi_i}{\varphi}U_{ii}, \quad \alpha\neq i
\end{eqnarray}
and
\begin{eqnarray}\label{2ex}
&&U_{ii\alpha\alpha}-U_{\alpha\alpha i i}
\nonumber \\ \nonumber&=&-\frac{\sum_{k}U_{\alpha\alpha k}\varphi_k}{\varphi}-\bigg[-\frac{2\varphi_{\alpha}^{2}}{\varphi^2}+\frac{|D \varphi|^2}{2\varphi^2}+\frac{1}{\varphi}U_{\alpha\alpha}+\frac{1}{2}\Big(1+\frac{1}{\varphi^2}\Big)\bigg]U_{\alpha\alpha}
\\&&+\frac{\sum_{k}U_{ii k}\varphi_k}{\varphi}+\bigg[-\frac{2\varphi_{i}^{2}}{\varphi^2}+\frac{|D \varphi|^2}{2\varphi^2}+\frac{1}{\varphi}U_{ii}+\frac{1}{2}\Big(1+\frac{1}{\varphi^2}\Big)\bigg]U_{ii}, \quad \alpha\neq i.
\end{eqnarray}
\end{lemma}

\begin{proof}
Using the Ricci identity on $\mathbb{S}^n$
\begin{eqnarray*}
\varphi_{kji}-\varphi_{kij}=\delta_{ik}\varphi_j-\delta_{jk}\varphi_i,
\end{eqnarray*}
we have
\begin{eqnarray*}
U_{\alpha\alpha i}&=&
\varphi_{\alpha \alpha i}-\frac{\sum_{l}\varphi_l \varphi_{li}}{\varphi}+\frac{1}{2}\frac{|D\varphi|^2\varphi_i}{\varphi^2}
+\frac{1}{2}\Big(\varphi_i+\frac{\varphi_i}{\varphi^2}\Big)
\\&=&\varphi_{i\alpha \alpha}-\varphi_i-\frac{\sum_{l}\varphi_l \varphi_{li}}{\varphi}+\frac{1}{2}\frac{|D\varphi|^2\varphi_i}{\varphi^2}
+\frac{1}{2}\Big(\varphi_i+\frac{\varphi_i}{\varphi^2}\Big)
\\&=&U_{i \alpha \alpha}-\frac{\varphi_i}{\varphi}U_{ii}
\end{eqnarray*}
and
\begin{eqnarray*}
U_{ii\alpha\alpha}&=&\varphi_{ii\alpha\alpha}
-\frac{\sum_{k}(\varphi^{2}_{k\alpha}+\varphi_{k\alpha\alpha}\varphi_k)}{\varphi}
+\frac{2\sum_{k}\varphi_{k\alpha}\varphi_k\varphi_{\alpha}}{\varphi^2}+\frac{1}{2}|D\varphi|^2
\frac{\varphi\varphi_{\alpha\alpha}-2\varphi_{\alpha}^{2}}{\varphi^3}
\\&&+\frac{1}{2}\Big(\varphi_{\alpha\alpha}+\frac{\varphi_{\alpha\alpha}}{\varphi^2}-
\frac{2\varphi_{\alpha}^{2}}{\varphi^3}\Big)
\\&=&\varphi_{ii\alpha\alpha}
-\frac{\sum_{k}(\varphi^{2}_{k\alpha}+\varphi_{\alpha\alpha k}\varphi_k)}{\varphi}
+\frac{\varphi_{\alpha}^{2}-|D \varphi|^2}{\varphi}
+\frac{2\sum_{k}\varphi_{k\alpha}\varphi_k\varphi_{\alpha}}{\varphi^2}\\&&+\frac{1}{2}|D\varphi|^2
\frac{\varphi\varphi_{\alpha\alpha}-2\varphi_{\alpha}^{2}}{\varphi^3}
+\frac{1}{2}\Big(\varphi_{\alpha\alpha}+\frac{\varphi_{\alpha\alpha}}{\varphi^2}-
\frac{2\varphi_{\alpha}^{2}}{\varphi^3}\Big).
\end{eqnarray*}
Using the definition of $U$, it yields
\begin{eqnarray*}
\frac{2\sum_{k}\varphi_{k\alpha}\varphi_k\varphi_{\alpha}}{\varphi^2}-
\frac{|D\varphi|^2\varphi_{\alpha}^{2}}{\varphi^3}+\varphi_{\alpha}^{2}\Big(\frac{1}{\varphi}-\frac{1}{\varphi^3}\Big)
=\frac{2\varphi_{\alpha}^{2}U_{\alpha\alpha}}{\varphi^2}
\end{eqnarray*}
and
\begin{eqnarray*}
-\frac{\sum_{k}\varphi^{2}_{k\alpha}}{\varphi}+\frac{1}{2}|D\varphi|^2
\frac{\varphi_{\alpha\alpha}}{\varphi^2}
=-\frac{\varphi_{\alpha\alpha}}{\varphi}U_{\alpha\alpha}+\frac{1}{2}\varphi_{\alpha\alpha}\Big(1-\frac{1}{\varphi^2}\Big).
\end{eqnarray*}
Thus,
\begin{eqnarray*}
U_{ii\alpha\alpha}&=&\varphi_{ii\alpha\alpha}
-\frac{\sum_{k}\varphi_{\alpha\alpha k}\varphi_k}{\varphi}
-\frac{|D \varphi|^2}{\varphi}
+\Big(\frac{2\varphi_{\alpha}^{2}}{\varphi^2}-\frac{\varphi_{\alpha\alpha}}{\varphi}\Big)U_{\alpha\alpha}
+\varphi_{\alpha\alpha}.
\end{eqnarray*}
Using the following Ricci identity on $\mathbb{S}^n$
\begin{eqnarray*}
\varphi_{ii\alpha\alpha}=\varphi_{\alpha\alpha ii}-2\varphi_{\alpha\alpha}+2\varphi_{ii}
\end{eqnarray*}
and noticing that
\begin{eqnarray*}
\varphi_{\alpha\alpha k}-\varphi_{ii k}=U_{\alpha\alpha k}-U_{ii k}, \quad
\varphi_{\alpha\alpha}-\varphi_{ii}=U_{\alpha\alpha}-U_{ii},
\end{eqnarray*}
we have
\begin{eqnarray*}
&&U_{ii\alpha\alpha}-U_{\alpha\alpha i i}\\&=&
-\frac{\sum_{k}U_{\alpha\alpha k}\varphi_k}{\varphi}+\frac{\sum_{k}U_{ii k}\varphi_k}{\varphi}
+\Big(\frac{2\varphi_{\alpha}^{2}}{\varphi^2}-\frac{\varphi_{\alpha\alpha}}{\varphi}\Big)U_{\alpha\alpha}
-\Big(\frac{2\varphi_{i}^{2}}{\varphi^2}-\frac{\varphi_{ii}}{\varphi}\Big)U_{ii}+U_{ii}-U_{\alpha\alpha}
\\&=&-\frac{\sum_{k}U_{\alpha\alpha k}\varphi_k}{\varphi}+\frac{\sum_{k}U_{ii k}\varphi_k}{\varphi}
-\bigg[-\frac{2\varphi_{\alpha}^{2}}{\varphi^2}+\frac{|D \varphi|^2}{2\varphi^2}+\frac{1}{\varphi}U_{\alpha\alpha}+\frac{1}{2}\Big(1+\frac{1}{\varphi^2}\Big)\bigg]U_{\alpha\alpha}
\\&&+\bigg[-\frac{2\varphi_{i}^{2}}{\varphi^2}+\frac{|D \varphi|^2}{2\varphi^2}+\frac{1}{\varphi}U_{ii}+\frac{1}{2}\Big(1+\frac{1}{\varphi^2}\Big)\bigg]U_{ii}.
\end{eqnarray*}
\end{proof}

The proof of Full Rank Theorem \ref{FRT} is based on the following Deformation Lemma.

\begin{lemma}\label{DL}
Let $O\subset \mathbb{S}^n$ be an open subset, suppose
$\varphi \in C^4(O)$ is an even solution of \eqref{MA} in $O$,
and that the matrix $U$ is positive semi-definite. Suppose there is a positive constant $C_0>0$ such that
for a fixed integer $(n-1)\geq l\geq k$, $\sigma_l(U[\varphi(x)])\geq C_0$ for all $x \in O$.
Let $\psi(x)=\sigma_{l+1}(U[\varphi(x)])$ and let $\tau(x)$ be the largest eigenvalue of
\begin{eqnarray*}
-(f^{-1})_{ij}+|D f^{-1}|\delta_{ij}
-\frac{f^{-1}}{\frac{8}{n}\max_{\mathbb{S}^n}f+2}\delta_{ij}.
\end{eqnarray*}
Then, there are constants $C_1, C_2$ depending only on
$|\varphi|_{C^3}$, $|f|_{C^{2}}$, $n$ and $C_0$ such that
the differential inequality
\begin{eqnarray*}
\Delta\psi(x)\leq
(n-l)\varphi^{-1}(x)f^{2}(x)\sigma_l(U(x))\tau(x)+C_1|\nabla \psi(x)|+C_2\psi(x)
\end{eqnarray*}
holds in $O$.
\end{lemma}

\begin{proof}
Following the notation of Caffarelli-Friedman \cite{CF85} and Guan-Ma \cite{GM03}.
For any two functions defined in an open set $O \subset \mathbb{S}^n$, $y \in
O$, we say that $h(y)\lesssim m(y)$ provided there exist positive
constants $c_1$ and $c_2$ such that
\begin{eqnarray}\label{cr-1}
h(y)-m(y)\leq c_1|\nabla \psi(y)|+c_2\psi(y).
\end{eqnarray}
We write $h(y)\sim m(y)$ if $h(y)\lesssim m(y)$ and $h(y)\lesssim m(y)$. Moreover,
we write $h\lesssim m$ if the inequality \eqref{cr-1} holds in $O$, with the constants
$c_1$ and $c_2$ depending only on $|\varphi|_{C^3}$,
$|f|_{C^{2}}$, $n$ and $C_0$ (independent of $y$ and $O$).
Finally, we write $h\sim m$ if $h\lesssim m$ and $h\lesssim m$. We shall show that
\begin{eqnarray}\label{cr-2}
\Delta\psi \lesssim (n-l)\varphi^{-1}f^{2}\sigma_l(U)\tau.
\end{eqnarray}

For any $x \in O$, let $\lambda_1\geq \lambda_2\geq ...\geq
\lambda_n$ be the eigenvalues of $U$ at $x$. Since $\sigma_l(U)\geq
C_0>0$, there exists a positive constant $C>0$ such that
$\lambda_1\geq \lambda_2\geq ...\geq \lambda_l\geq C$. Let
\begin{eqnarray*}
G=\{1, 2, ..., l\} \quad \mbox{and} \quad B=\{l+1, 2, ..., n\}
\end{eqnarray*}
be the ``good" and ``bad" sets of indices respectively. We denote
by $\sigma_k(U|i)$ the $k$-th element symmetric function of $U$ excluding
the $i$-column and $i$-row and $\sigma_k(U|ij)$ the $k$-th element symmetric function of $U$ excluding
the $i, j$-column and $i, j$-row. Let $\Lambda_{G}=(\lambda_1, ..., \lambda_l)$ be
the ``good" eigenvalues of $U$ at $z$; for convenience in notation,
we also write $G=\Lambda_G$ if there is no confusion. In the following, all calculations are
at the point $x$ using the relation ``$\lesssim$", with the understanding that the constants
in \eqref{cr-1} are under control.

For each fixed $x \in O$, we choose a local orthonormal frame $e_1,
..., e_n$ so that $\psi$ is diagonal at $x$, and $U_{ii}=\lambda_i$ for
$i=1, 2, ..., n$. Now we compute $\psi$ and its first and second
derivatives in the direction $e_{\alpha}$.

As $\psi=\sigma_{l+1}(U)=\frac{1}{l+1}\sum_{i=1}^{n}\sigma_{l}(U|i)U_{ii}$, we find that
\begin{eqnarray*}
0\sim \psi(z)\sim \Big(\sum_{i \in B}U_{ii}\Big)\sigma_l(G)\sim
\sum_{i \in B}U_{ii},
\end{eqnarray*}
so
\begin{eqnarray}\label{CR-0}
U_{ii}\sim 0 \quad \mbox{for all} \quad i \in B,
\end{eqnarray}
which yields that, for $1\leq m\leq l$,
\begin{eqnarray}\label{CR-1}
\sigma_m(U)\sim \sigma_m(G), \quad
\sigma_{m}(U|j)\sim
  \begin{cases}
    \displaystyle \sigma_{m}(G|j), & \text{if}\quad j \in G;\\[2.5ex]
    \displaystyle \sigma_{m}(G), & \text{if} \quad j \in B.
  \end{cases}
\end{eqnarray}
\begin{equation}\label{CR-2}
 \sigma_{m}(U|ij)\sim
  \begin{cases}
    \displaystyle \sigma_{m}(G|ij), & \text{if}\quad i, j \in G;\\[2.5ex]
    \displaystyle \sigma_{m}(G|j), & \text{if} \quad i \in B, j \in G;\\[2.5ex]
    \displaystyle \sigma_{m}(G), & \text{if} \quad i,j \in B, i\neq
    j.
  \end{cases}
\end{equation}
Moreover, $\psi_{\alpha}=\sigma_{l+1}^{ij}U_{ij\alpha}
=\sum_{j \in B}\sigma_{l}(U|j)U_{jj\alpha}+\sum_{j \in G}\sigma_{l}(U|j)U_{jj\alpha}$, then \eqref{CR-1} tells us that
\begin{eqnarray}\label{CR-4}
0\sim \psi_{\alpha}\sim \Big(\sum_{i \in
B}U_{ii\alpha}\Big)\sigma_l(G)\sim \sum_{i \in B}U_{ii\alpha}.
\end{eqnarray}

By Proposition \ref{AA}, we have
\begin{eqnarray}\label{CRR-1}
\sigma_{l+1}^{ij}\sim
  \begin{cases}
    \displaystyle \sigma_{l}(G), & \text{if}\quad i=j \in G;\\[2.5ex]
    \displaystyle 0, & \text{otherwise}.
  \end{cases}
\end{eqnarray}

\begin{equation}\label{CRR-2}
 \sigma_{l+1}^{ij, rs}=
  \begin{cases}
    \displaystyle \sigma_{l-1}(U|ir), & \text{if}\quad i=j, r=s, i\neq r;\\[2.5ex]
    \displaystyle -\sigma_{l-1}(U|ij), & \text{if} \quad i\neq j, r=j, s=i;\\[2.5ex]
    \displaystyle 0, & \text{otherwise}.
  \end{cases}
\end{equation}

Now, we will use \eqref{CR-0}-\eqref{CRR-2} to single out the main terms
in the calculation of $\Delta \psi$.
Since
\begin{eqnarray*}
\psi_{\alpha\alpha}=\sigma_{l+1}^{ij,
r s}U_{ij\alpha}U_{rs\alpha}+\sigma_{l+1}^{ij}U_{ij\alpha\alpha},
\end{eqnarray*}
it follows from \eqref{CRR-2} and \eqref{1ex-0} that for any $\alpha
\in \{1, 2, ..., n\}$
\begin{eqnarray}\label{4.11}
\psi_{\alpha\alpha}&=&\sum_{i\neq
j}\sigma_{l-1}(U|i j)U_{ii\alpha}U_{jj\alpha}-\sum_{i\neq
j}\sigma_{l-1}(U|i j)U^{2}_{ij\alpha}+\sum_{i=1}^{n}\sigma_{l+1}^{ii}W_{ii\alpha\alpha}
\nonumber\\&=&\bigg(\sum_{i \in G,
j \in B}+\sum_{i \in B, j \in G}+\sum_{i, j \in B, i\neq j}+\sum_{i,
j \in G, i\neq
j}\bigg)\sigma_{l-1}(U|i j)U_{ii\alpha}U_{jj\alpha}\nonumber\\&&-\bigg(\sum_{i
\in G, j \in B}+\sum_{i \in B, j \in G}+\sum_{i, j \in B, i\neq
j}+\sum_{i, j \in G, i\neq
j}\bigg)\sigma_{l-1}(U|i j)U^{2}_{ij\alpha}\nonumber\\&&+\sum_{i=1}^{n}\sigma_{l+1}^{ii}U_{ii\alpha\alpha}.
\end{eqnarray}
Using \eqref{CR-2} and \eqref{CR-4}, we find
\begin{eqnarray}\label{4.12}
\sum_{i \in B, j \in G}\sigma_{l-1}(U|i j)U_{ii\alpha}U_{jj\alpha}
\sim \bigg(\sum_{j \in G}\sigma_{l-1}(G|j)U_{jj\alpha}\bigg)\sum_{i \in B} U_{ii\alpha}
\sim 0.
\end{eqnarray}
We know from \eqref{CR-4} that
\begin{eqnarray}\label{4.12-1}
-U_{ii \alpha}\sim \sum_{j \in B, j\neq i}U_{jj \alpha} \quad \forall i \in B, \ \forall \alpha.
\end{eqnarray}
Combining \eqref{4.12-1} and \eqref{CR-2} yields
\begin{eqnarray}\label{4.13}
\sum_{i, j\in B, i\neq j}\sigma_{l-1}(U|i j)U_{ii\alpha}U_{jj\alpha}
\sim -\sigma_{l-1}(G)\sum_{i \in B}U^{2}_{ii\alpha}.
\end{eqnarray}
Moreover, \eqref{CR-2} gives
\begin{eqnarray}\label{4.14}
\sum_{i, j\in B, i\neq j}\sigma_{l-1}(U|i j)U_{ii\alpha}U_{jj\alpha}
\sim -\sigma_{l-1}(G)\sum_{i \in B}U^{2}_{ii\alpha}
\end{eqnarray}
Plugging \eqref{4.12}, \eqref{4.13} and \eqref{4.14} into \eqref{4.11}, we obtain
by \eqref{CR-2}
\begin{eqnarray*}
\psi_{\alpha\alpha}\sim\sum_{i=1}^{n}\sigma_{l+1}^{ii}U_{ii\alpha\alpha}-2\sum_{i
\in B, j \in
G}\sigma_{l-1}(G|j)U^{2}_{ij\alpha}-\sigma_{l-1}(G)\sum_{i, j \in
B}U^{2}_{ij\alpha}.
\end{eqnarray*}
Summing $\psi_{\alpha\alpha}$ with $\alpha$ from $1$ to $n$, it yields
\begin{eqnarray}\label{MM-1}
\Delta\psi&\sim&\sum_{i=1}^{n}\sigma_{l+1}^{ii}\sum_{\alpha=1}^{n}U_{ii\alpha\alpha}-2\sum_{i
\in B, j \in
G}\sigma_{l-1}(G|j)\sum_{\alpha=1}^{n}U^{2}_{ij\alpha}\nonumber\\&&-\sigma_{l-1}(G)\sum_{i,
j \in B}\sum_{\alpha=1}^{n}U^{2}_{ij\alpha}.
\end{eqnarray}

As $\psi=\sigma_{l+1}(U)=\frac{1}{l+1}\sum_{i=1}^{n}\sigma_{l+1}^{ii}U_{ii}$ and
$\psi_{k}=\sigma_{l+1}^{ij}U_{ijk}$, we find that for each fixed $i$
\begin{eqnarray}\label{MM-2}
0\sim \psi\sim \sigma_{l+1}^{ii}U_{ii}.
\end{eqnarray}
\begin{eqnarray}\label{MM-3}
0\sim \psi_{k}=\sum_{i=1}^{n}\sigma_{l+1}^{ii}U_{iik}.
\end{eqnarray}
Inserting \eqref{MM-2} and \eqref{MM-3} into \eqref{2ex}, we obtain
\begin{eqnarray*}
&&\sigma_{l+1}^{ii}\sum_{\alpha=1}^{n}(U_{ii\alpha\alpha}-U_{\alpha\alpha i i})
\\&\sim &-\frac{\sum_{i=1}^{n}\sigma_{l+1}^{ii}\sum_{\alpha, k=1}^{n}U_{\alpha\alpha k}\varphi_k}{\varphi}
\\&&-\sum_{i, \alpha=1}^{n}\sigma_{l+1}^{ii}\bigg[-\frac{2\varphi_{\alpha}^{2}}{\varphi^2}+\frac{|D \varphi|^2}{2\varphi^2}+\frac{1}{\varphi}U_{\alpha\alpha}+\frac{1}{2}\Big(1+\frac{1}{\varphi^2}\Big)\bigg]U_{\alpha\alpha}
\\&\sim &-\frac{(n-l)\sigma_l(G)\sum_{k=1}^{n}(\varphi^{-1}f)_{k}\varphi_k}{\varphi}
\\&&-(n-l)\sigma_l(G)\sum_{\alpha=1}^{n}\bigg[-\frac{2\varphi_{\alpha}^{2}}{\varphi^2}+\frac{|D \varphi|^2}{2\varphi^2}+\frac{1}{\varphi}U_{\alpha\alpha}+\frac{1}{2}\Big(1+\frac{1}{\varphi^2}\Big)\bigg]U_{\alpha\alpha}.
\end{eqnarray*}
Moreover, we have by the equation \eqref{MA}
\begin{eqnarray*}
\sum_{i=1}^{n}\sigma_{l+1}^{ii}\sum_{\alpha=1}^{n}U_{\alpha\alpha i i}
=\sum_{i=1}^{n}\sigma_{l+1}^{ii}(\varphi^{-1}f)_{i i}
=\sigma_l(G)\sum_{i \in B}(\varphi^{-1}f)_{i i}.
\end{eqnarray*}
Then, it follows that
\begin{eqnarray}\label{MC-1}
&&\sigma_{l+1}^{ii}\sum_{\alpha}U_{ii\alpha\alpha}
\nonumber\\ \nonumber&\sim &\sigma_l(G)\sum_{i \in B}(\varphi^{-1}f)_{i i}-\frac{(n-l)\sigma_l(G)\sum_{k=1}^{n}(\varphi^{-1}f)_{k}\varphi_k}{\varphi}
\\&&-(n-l)\sigma_l(G)\sum_{\alpha=1}^{n}\bigg[-\frac{2\varphi_{\alpha}^{2}}{\varphi^2}+\frac{|D \varphi|^2}{2\varphi^2}+\frac{1}{\varphi}U_{\alpha\alpha}+\frac{1}{2}\Big(1+\frac{1}{\varphi^2}\Big)\bigg]U_{\alpha\alpha}.
\end{eqnarray}
Taking \eqref{MM-1} into \eqref{MC-1}, we find
\begin{eqnarray}\label{MC-2}
\Delta\psi&\sim&\sigma_l(G)\sum_{i \in B}(\varphi^{-1}f)_{i i}
-\frac{(n-l)\sigma_l(G)\sum_{k=1}^{n}(\varphi^{-1}f)_{k}\varphi_k}{\varphi}
\nonumber\\ \nonumber&&-(n-l)\sigma_l(G)\sum_{\alpha=1}^{n}\bigg[-\frac{2\varphi_{\alpha}^{2}}{\varphi^2}+\frac{|D \varphi|^2}{2\varphi^2}+\frac{1}{\varphi}U_{\alpha\alpha}+\frac{1}{2}\Big(1+\frac{1}{\varphi^2}\Big)\bigg]U_{\alpha\alpha}\\&&-2\sum_{i
\in B, j \in
G}\sigma_{l-1}(G|j)\sum_{\alpha=1}^{n}U^{2}_{ij\alpha}-\sigma_{l-1}(G)\sum_{i,
j \in B}\sum_{\alpha=1}^{n}U^{2}_{ij\alpha}.
\end{eqnarray}
Now, we deal with the fourth term on the right hand side of \eqref{MC-2}. Using \eqref{1ex-0} and \eqref{1ex}, we obtain
\begin{eqnarray}\label{MC-3}
&&-2\sum_{\alpha=1}^{n}\sum_{i
\in B, j \in
G}\sigma_{l-1}(G|j)U^{2}_{ij\alpha}\nonumber\\&\leq&-2\sum_{i
\in B, \alpha \in
G}\sigma_{l-1}(G|\alpha)U^{2}_{i\alpha\alpha}-2\sum_{i
\in B, \alpha \in
G}\sigma_{l-1}(G|\alpha)U^{2}_{i\alpha i}\nonumber\\&\lesssim&
-2\sum_{i
\in B}\sum_{\alpha \in
G}\sigma_{l-1}(G|\alpha)U^{2}_{i\alpha\alpha}-2\sum_{i
\in B, \alpha \in
G}\sigma_{l-1}(G|\alpha)U^{2}_{\alpha\alpha }\frac{\varphi_{\alpha}^{2}}{\varphi^2}\nonumber\\&\lesssim&
-2\sum_{i
\in B}\sum_{\alpha \in
G}\sigma_{l-1}(G|\alpha)U^{2}_{\alpha\alpha i}-2(n-l)\sigma_l(G)\sum_{\alpha \in
G}U_{\alpha\alpha }\frac{\varphi_{\alpha}^{2}}{\varphi^2},
\end{eqnarray}
where we used the following relation
\begin{eqnarray*}
U_{i\alpha\alpha}=U_{\alpha\alpha i}+\frac{\varphi_i}{\varphi}U_{ii}\sim U_{\alpha\alpha i}, \quad i \in B, \alpha \in G
\end{eqnarray*}
to get the last inequality. Since
\begin{eqnarray*}
&&\sum_{\alpha \in
G}\sigma_{l-1}(G|\alpha)U^{2}_{\alpha\alpha i}
-\sigma_{l}(G)\frac{(\varphi^{-1}f)_{i}^{2}}{\varphi^{-1}f}
\\&=&\sum_{\alpha \in
G}\sigma_{l-1}(G|\alpha)U^{2}_{\alpha\alpha i}
-\sigma_{l}(G)\frac{(\sum_{\alpha=1}^{n}U_{\alpha\alpha i})^{2}}{\varphi^{-1}f}
\\&\sim& \sum_{\alpha \in
G}\sigma_{l-1}(G|\alpha)U^{2}_{\alpha\alpha i}
-\frac{\Big[\sum_{\alpha \in G}\sigma_{l-1}^{\frac{1}{2}}(G|\alpha)U^{\frac{1}{2}}_{\alpha\alpha}U_{\alpha\alpha i}
\Big]^{2}}{\varphi^{-1}f}
\\&\geq& \sum_{\alpha \in
G}\sigma_{l-1}(G|\alpha)U^{2}_{\alpha\alpha i}
-\frac{\sum_{\alpha \in G}\sigma_{l-1}(G|\alpha)U^{2}_{\alpha\alpha i}\sum_{\beta \in G}U_{\beta\beta}
}{\varphi^{-1}f}
\\&\sim& 0,
\end{eqnarray*}
the inequality \eqref{MC-3} becomes
\begin{eqnarray*}
&&-2\sum_{\alpha=1}^{n}\sum_{i
\in B, j \in
G}\sigma_{l-1}(G|j)U^{2}_{ij\alpha}
\\&\lesssim&
-2\sum_{i
\in B}\sigma_{l}(G)\frac{(\varphi^{-1}f)_{i}^{2}}{\varphi^{-1}f}-2(n-l)\sigma_l(G)\sum_{\alpha \in
G}U_{\alpha\alpha }\frac{\varphi_{\alpha}^{2}}{\varphi^2}.
\end{eqnarray*}
Putting the above inequality into \eqref{MC-2}, it yields
\begin{eqnarray*}
\Delta\psi&\lesssim&\sigma_l(G)\sum_{i \in B}(\varphi^{-1}f)_{i i}-\frac{(n-l)\sigma_l(G)\sum_{k=1}^{n}(\varphi^{-1}f)_{k}\varphi_k}{\varphi}
\\&&-(n-l)\sigma_l(G)\sum_{\alpha=1}^{n}\bigg[\frac{|D \varphi|^2}{2\varphi^2}+\frac{1}{\varphi}U_{\alpha\alpha}+\frac{1}{2}\Big(1+\frac{1}{\varphi^2}\Big)\bigg]U_{\alpha\alpha}\\&&-2\sum_{i
\in B}\sigma_{l}(G)\frac{(\varphi^{-1}f)_{i}^{2}}{\varphi^{-1}f}.
\end{eqnarray*}
Note that
\begin{eqnarray*}
\varphi^{-1}f[(\varphi^{-1}f)_{i}]^2=\Big(-\frac{\varphi_i f}{\varphi^2}+\frac{f_i}{\varphi}\Big)^2
=\frac{\varphi_{i}^{2} f}{\varphi^2}+\frac{f_{i}^{2}}{f \varphi}-\frac{2\varphi_if_{i}}{\varphi^2},
\end{eqnarray*}
\begin{eqnarray*}
(\varphi^{-1}f)_{ii}&=&-\frac{\varphi_{ii}}{\varphi}\varphi^{-1}f+\frac{2\varphi_i\varphi_if}{\varphi^3}
+\frac{f_{ii}}{\varphi}-\frac{2\varphi_i f_i}{\varphi^2}
\\&\sim& -\frac{|D \varphi|^2}{2\varphi^2}\varphi^{-1}f+\frac{1}{2}(1-\frac{1}{\varphi^2})\varphi^{-1}f
+\frac{2\varphi^{2}_{i}}{\varphi^2}\varphi^{-1}f+\frac{f_{ii}}{\varphi}-\frac{2\varphi_i f_i}{\varphi^2},
\end{eqnarray*}
and
\begin{eqnarray*}
-\frac{\sum_{k=1}^{n}(\varphi^{-1}f)_{k}\varphi_k}{\varphi}=\frac{f|D\varphi|^2}{\varphi^3}
-\frac{\sum_{k=1}^{n}f_{k}\varphi_k}{\varphi^2}.
\end{eqnarray*}
Thus,
\begin{eqnarray*}
\Delta\psi&\lesssim&\sigma_l(G)\sum_{i \in B}\Big(\frac{f_{ii}}{\varphi}-\frac{2f_i\varphi_i}{\varphi^2}\Big) -\frac{(n-l)\sigma_l(G)\sum_{k=1}^{n}f_{k}\varphi_k}{\varphi^2}
\\&&-2\sum_{i \in B}\sigma_l(G)\Big(\frac{f_{i}^{2}}{f \varphi}-\frac{2\varphi_if_{i}}{\varphi^2}\Big)-(n-l)\sigma_l(G)\frac{f}{\varphi^3}
\\&\lesssim&\frac{\sigma_l(G)}{\varphi}\sum_{i \in B}\bigg(f_{ii}-\frac{\sum_{\alpha \in G}f_{\alpha}\varphi_\alpha-\sum_{k \in B}f_{k}\varphi_k}{\varphi}
-2\frac{f_{i}^{2}}{f}-\frac{f}{\varphi^2}\bigg)\\&\lesssim&\frac{\sigma_l(G)}{\varphi}\sum_{i \in B}\bigg(f_{ii}+|D f|
-2\frac{f_{i}^{2}}{f}-\frac{f}{\frac{8}{n}\max_{\mathbb{S}^n}f+2}\bigg)\\&\lesssim&-f^2\frac{\sigma_l(G)}{\varphi}\sum_{i \in B}
\bigg((f^{-1})_{ii}-|D f^{-1}|
+\frac{f^{-1}}{\frac{8}{n}\max_{\mathbb{S}^n}f+2}\bigg),
\end{eqnarray*}
where we used the inequalities \eqref{C0} and \eqref{r-u} to get the last but one inequality.
This complete the proof.
\end{proof}

Now, we begin to prove Theorem \ref{FRT}.

\begin{proof}
If $\varphi$ is not uniformly $h$-convex at some point $x_0$, then
there is an integer $l$ with $1\leq l\leq n-1$ such that $\sigma_l(U[\varphi])>0$
for any $x \in \mathbb{S}^n$ and $\psi(x_0)=\sigma_{l+1}(U[\varphi(x_0)])=0$.
Using Lemma \ref{DL}, we have
\begin{eqnarray*}
\Delta\psi(x)\leq C_1|\nabla \psi(x)|+C_2\psi(x).
\end{eqnarray*}
The strong minimum principle implies $\psi=\sigma_{l+1}(U[\varphi])\equiv 0$. Then,
we have $\varphi\equiv 1$ by \eqref{DE}. This is a contradiction to \eqref{MA}.
\end{proof}

\section{The proof of the main theorem}

In this section, we use the degree theory for nonlinear elliptic
equations developed in \cite{Li89} to prove Theorem \ref{Main}.
For the use of the degree theory, the uniqueness of constant solutions to the equation
\eqref{MA} is important for us.

\begin{lemma}\label{U-C}
The $h$-convex solutions to the equation
\begin{eqnarray}\label{MA-c}
\Delta\varphi-\frac{n}{2}\frac{|D\varphi|^2}{\varphi}+\frac{n}{2}\Big(\varphi-\frac{1}{\varphi}\Big)=\varphi^{-1}\gamma
\end{eqnarray}
with $\varphi>1$ are given by
\begin{eqnarray*}
\varphi(x)=\Big(1+\frac{2}{n}\gamma\Big)^{\frac{1}{2}}\Big(\sqrt{|x_0|^2+1}-\langle x_0, x\rangle\Big),
\end{eqnarray*}
where $x_0 \in \mathbb{R}^{n+1}$. In particular, $\varphi(x)=\Big(1+\frac{2}{n}\gamma\Big)^{\frac{1}{2}}$ is the unique even solution.
\end{lemma}

\begin{proof}
This lemma is a corollary of Proposition 8.1 in \cite{Li-Xu}, its proof is similar to that of Theorem 8.1 (7).
\end{proof}

We define
\begin{eqnarray*}
\mathcal{B}^{2,\alpha}(\mathbb{S}^n)=\{\varphi \in
C^{2,\alpha}(\mathbb{S}^n): \varphi \ \mbox{is even}\}
\end{eqnarray*}
and
\begin{eqnarray*}
\mathcal{B}_{0}^{4,\alpha}(\mathbb{S}^n)=\{\varphi \in
C^{4,\alpha}(\mathbb{S}^n): U[\varphi]>0 \ \mbox{and} \ \varphi \ \mbox{is even}\}.
\end{eqnarray*}
Let us consider $$\mathcal{L}(\cdot, t): \mathcal{B}_{0}^{4,\alpha}(\mathbb{S}^n)\rightarrow
\mathcal{B}^{2,\alpha}(\mathbb{S}^n),$$ which is defined by
\begin{eqnarray*}
\mathcal{L}(\varphi, t)=\Delta\varphi-\frac{n}{2}\frac{|D\varphi|^2}{\varphi}+\frac{n}{2}
\Big(\varphi-\frac{1}{\varphi}\Big)-\varphi^{-1} f_t,
\end{eqnarray*}
where
\begin{eqnarray}\label{ft}
f_t=[(1-t)(\max_{\mathbb{S}^n}f)^{-1}+t f^{-1}]^{-1}.
\end{eqnarray}

Let $$\mathcal{O}_R=\{\varphi \in \mathcal{B}_{0}^{4,\alpha}(\mathbb{S}^n):
1+\frac{1}{R}< \varphi, \ 0< U[\varphi], \ |\varphi|_{C^{4,\alpha}(\mathbb{S}^n)}<R\},$$ which clearly is an open
set of $\mathcal{B}_{0}^{4,\alpha}(\mathbb{S}^n)$. Moreover, if $R$ is
sufficiently large, $\mathcal{L}(\varphi, t)=0$ has no solution on $\partial
\mathcal{O}_R$ by Full Rank Theorem \ref{FRT}, and the a priori estimates \eqref{C0} and \eqref{C2+}.
Otherwise, we can find a family of solutions $\{\varphi_i(\cdot, t_i)\} \subset\mathcal{O}_R$
satisfying $\mathcal{L}(\varphi, t)=0$ and $x_i \in \mathbb{S}^n$ satisfying
$\det (U[\varphi_i(x_i, t_i)])\rightarrow 0$ as $i\rightarrow +\infty$. Then, there is a
subsequences $\{\varphi_{i_m}\}$ and $x_{i_m}$ satisfying $\varphi_{i_m}(\cdot, t_{i_m})\rightarrow \varphi$
in $C^{4, \alpha}(\mathbb{S}^n)$, $t_{i_m}\rightarrow t_0$ and $x_{i_m}\rightarrow x_0$ as $m\rightarrow +\infty$ .
Moreover, $\varphi$ is a $h$-convex solution to $\mathcal{L}(\varphi, t_0)=0$
satisfying $\det(U[\varphi(x_0)])=0$ and \eqref{DE}.
On the other hand, it is easy to check that $f_t$ satisfies
Conditions (1) and (2) in Theorem \ref{Main}, this is a contradiction due to Full Rank Theorem \ref{FRT}.

Therefore the degree $\deg(\mathcal{L}(\cdot, t), \mathcal{O}_R, 0)$ is
well-defined for $0\leq t\leq 1$. Using the homotopic invariance of
the degree (Proposition 2.2 in \cite{Li89}), we have
\begin{eqnarray}\label{hot}
\deg(\mathcal{L}(\cdot, 1), \mathcal{O}_R, 0)=\deg(\mathcal{L}(\cdot, 0), \mathcal{O}_R, 0).
\end{eqnarray}

Lemma \ref{U-C} tells us that
$\varphi=c$ is the unique even solution for $\mathcal{L}(\varphi, 0)=0$ in $\mathcal{O}_R$.
Direct calculation show that the linearized operator of $\mathcal{L}$ at
$\varphi=c$ is
\begin{eqnarray*}
L_{c}(\overline{\phi})=(\Delta_{\mathbb{S}^n}+n)\overline{\phi}.
\end{eqnarray*}
Since
$\Delta_{\mathbb{S}^n}\overline{\phi}+n\overline{\phi}=0$
has the unique even solution $\overline{\phi}=0$,
$L_{c}$ is an invertible operator. So, we have by Proposition
2.3 in \cite{Li89}
\begin{eqnarray*}
\deg(\mathcal{L}(\cdot, 0), \mathcal{O}_R, 0)=\deg(L_{c_0}, \mathcal{O}_R, 0).
\end{eqnarray*}
Because
the eigenvalues of the Beltrami-Laplace operator $\Delta$ on $\mathbb{S}^n$ are strictly less than
$-n$ except for the first two eigenvalues $0$ and $-n$,
there is only one positive eigenvalue $n$ of $L_{c}$
with multiplicity $1$.
Then we have by Proposition
2.4 in \cite{Li89}
\begin{eqnarray*}
\deg(\mathcal{L}(\cdot, 0), \mathcal{O}_R, 0)=\deg(L_{c_0}, \mathcal{O}_R, 0)=-
1.
\end{eqnarray*}
Therefore, it follows from \eqref{hot}
\begin{eqnarray*}
\deg(\mathcal{L}(\cdot, 1), \mathcal{O}_R; 0)=\deg(\mathcal{L}(\cdot, 0), \mathcal{O}_R, 0)=-
1.
\end{eqnarray*}
So, we obtain a solution at $t=1$. This completes the proof of
Theorem \ref{Main}.

\end{document}